\documentclass[10pt]{article}
\usepackage{times}               
\usepackage{amsmath,amssymb,amsthm,url}
\def\d{d}
\def\la{\lambda}%
\def\a{\alpha}
\def\b{\beta}
\def\ga{\gamma}
\def\Ph{\Phi_{\la}}

\newcommand{\deriv}[3][]{\frac{\d^{#1}{#2}}{{\d{#3}}^{#1}}}
\def\w{w}
\def\v{v}
\def\ww#1{w(#1;t)}
\def\vvx{\deriv{\v}{x}(x;t)} 
\def\vvy{\deriv{\v}{y}(y;t)} 
\def\vx{v'(x;t)}
\def\vy{v'(y;t)}
\def\vvxy{\mathcal{K}(x,y)}

\def\W{\mathcal{W}}

\def\N{\mathbb{N}}
\def\R{\mathbb{R}}

\def\PI{\hbox{\rm P$_{\rm I}$}}
\def\PII{\hbox{\rm P$_{\rm II}$}}

\def\PIV{\hbox{\rm P$_{\rm IV}$}}

\def\sPIV{\hbox{\rm S$_{\rm IV}$}}

\def\dPI{\hbox{\rm dP$_{\rm I}$}}

\newcommand{\WhitD}[1]{D_{#1}}

\newtheorem{theorem}{Theorem}[section]
\newtheorem{proposition}[theorem]{Proposition}
\newtheorem{lemma}[theorem]{Lemma}
\newtheorem{corollary}[theorem]{Corollary}
\theoremstyle{definition}

\newtheorem{remark}[theorem]{Remark}
\newtheorem{remarks}[theorem]{Remarks}
\numberwithin{figure}{section}
\numberwithin{equation}{section}
\numberwithin{table}{section}

\def\ds{\displaystyle}
\def\erf{\mathop{\rm erf}\nolimits}
\def\erfc{\mathop{\rm erfc}\nolimits}

\numberwithin{figure}{section}
\numberwithin{equation}{section}
\numberwithin{table}{section}

\newcommand{\comment}[1]{}
\def\p{Painlev\'{e}}

\def\peqs{Painlev\'{e} equations}

\begin{document}
\title{A Generalized Freud Weight}
\author{Peter A.\ Clarkson\\ School of Mathematics, Statistics and Actuarial Science,\\
University of Kent, Canterbury, CT2 7NF, UK\\ 
\texttt{P.A.Clarkson@kent.ac.uk}\\[10pt]
Kerstin Jordaan 
and 
Abey Kelil\\ Department of Mathematics and Applied Mathematics,\\
University of Pretoria, Pretoria, 0002, South Africa\\ 
\texttt{kerstin.jordaan@up.ac.za}, \texttt{abeysh2001@gmail.com}}
\date{}

\maketitle

\begin{abstract}
We discuss the relationship between the recurrence coefficients of orthogonal
polynomials with respect to a generalized Freud weight 
\[\ww{x}=|x|^{2\la+1}\exp\left(-x^4+tx^2\right),\qquad x\in\R,\]
with parameters $\la>-1$ and $t\in\R$,
and classical solutions of the fourth \p\ equation. We show that the coefficients in these recurrence
relations can be expressed in terms of Wronskians of parabolic cylinder functions that
arise in the description of special function solutions of the fourth \p\ equation.
Further we derive a second-order linear ordinary differential equation and a differential-difference
equation satisfied by the generalized Freud polynomials.
\end{abstract}
\section{Introduction}
Let $\mu$ be a positive Borel measure defined on the real line for which all the moments
$$\mu_n=\int_{\R} x^n\, d\mu(x), \quad n\in\N,$$ exist.
The Hilbert space $L^2(\mu)$ contains the set of polynomials and hence Gram-Schmidt orthogonalization applied to the set $\{1,x,x^2,\dots\}$ yields a set of monic orthogonal polynomials on the real line defined by
\begin{equation*}
\int_{\R}  P_n(x) P_m(x) \,d\mu(x)=h_n\delta_{mn}, \qquad m,n\in\N,
\end{equation*}
with $h_n>0$, $\delta_{mn}$ the Kronecker delta and $P_n(x)$ is a polynomial of exact degree $n$.

A family $\{P_n(x)\}_{n=0}^{\infty}$ of monic orthogonal polynomials satisfies a three-term recurrence relation of the form \begin{equation}
\label{3trr1} xP_n(x) =  P_{n+1}(x) + \a_n P_n(x) + \b_n P_{n-1}(x), \quad n\in\N,
\end{equation} 
with $P_{-1}(x)=0$ and $P_0(x)=1$. 
\comment{Comparing the leading coefficients in \eqref{3trr1} gives $\ds{a_{n+1}={\gamma_n}/{\gamma_{n+1}}>0,}$ while computing the Fourier coefficients of $xP_n(x)$ yields}%
Respectively multiplying \eqref{3trr1} by $P_n(x)$ and $P_{n-1}(x)$ and integrating gives
\begin{align*}
\a_n=\frac{1}{h_n}\int_{\R}  x \,P_n^2(x)\,d\mu(x),\qquad
\b_n&=\frac{1}{h_{n-1}}\int_{\R}  x\,P_n(x)P_{n-1}(x)\,d\mu(x).
\end{align*}
The converse statement, known as the spectral theorem for orthogonal polynomials, is often attributed to Favard \cite{refFavard} but was seemingly discovered independently, around the same time, by both Shohat \cite{refShohat36,refShohat38} and Natanson \cite{refNatanson}. It is interesting to note that the result can be traced back to earlier work on continued fractions with a rudimentary form given already in 1894 by Stieltjes \cite{refStieltjes}, see \cite{refChihara}, also \cite{refMarcellan,refwalter}. The result also appears in the books by Stone \cite{refStone} and Wintner \cite{refWintner}, see \cite{refIsmail}. A modern proof of the result is given by Beardon \cite{refBeardon}.

\begin{theorem}
If a family of polynomials  satisfies a three-term recurrence relation of the form
\begin{equation*}
xP_n(x) = P_{n+1}(x) + \a_n P_n(x) + \b_n P_{n-1}(x),
\end{equation*}
with initial conditions $P_{-1}(x)=0$ and $P_0(x)=1$, where $\a_n,\b_n\in\R$, then there exists a measure $\mu$ on the real line such that these polynomials are monic orthogonal polynomials satisfying
\begin{equation*}
\int_{\R}  P_n(x) P_m(x)\,d\mu(x)=h_n\delta_{mn}, \qquad m,n \in\N, 
\end{equation*} with $h_n>0$.
\end{theorem} 
\begin{proof}See, for example, \cite{refChihara,refIsmail}.
\end{proof}
Two important problems arise in the study of orthogonal polynomials.

\begin{enumerate}
\item \textit{Direct problem:} Given a measure $\mu(x)$, what can be deduced about the recurrence coefficients $\{\a_n,\b_n\}$, $n\in\N$?
\item \textit{Inverse problem:} Given recurrence coefficients $\{\a_n,\b_n\}$, $n\in\N$, what can be deduced about the orthogonality measure $\mu(x)$?
\end{enumerate}

Given an orthogonality measure $\mu(x)$, several characteristic properties of the sequence $\{P_n(x)\}_{n=0}^{\infty}$ are determined by the nature of the measure. Extracting this information from the measure is an extension of the direct problem and one of the interesting and challenging problems in the study of systems of orthogonal polynomials. Properties typically studied include the Hankel determinants, the coefficients of the three-term recurrence relation, the coefficients of the differential-difference equation and the differential equation satisfied by the polynomials. For instance, the recurrence coefficients can be expressed in terms of Hankel determinants comprising the moments of the orthogonality measure. We refer to \cite{refChihara,refIsmail,refSzego} for further information about orthogonal polynomials.

For classical orthogonal polynomials, namely the Hermite, Laguerre and Jacobi polynomials, the measure $\mu(x)$ is absolutely continuous and can be expressed in terms of a weight function $w(x)$ which is non-negative, with support on some interval $[a,b] \in\R$ (where $a=-\infty$ and $b=\infty$ are allowed), i.e.\ $d\mu(x)=w(x)\,dx$.
The properties that these orthogonal polynomials satisfy are well known and include the fact that (cf.~\cite{refIsmail}):
\begin{itemize}
\item[(a)] their derivatives also form orthogonal polynomial sets;
\item[(b)] they satisfy a Rodrigues'  type formula
\[P_n(x)=\frac{1}{\kappa_n w(x)}\deriv[n]{}{x}\left\{w(x)\sigma^n(x)\right\},\]
where $w(x)$ is the weight function, $\sigma(x)$ is a polynomial in $x$ independent of $n$ and $\kappa_n$ a constant;
\item[(c)] they satisfy a non-linear equation of the form 
\[\begin{split}\sigma(x)\deriv{}{x}\big[{P_n(x)P_{n-1}(x)}\big]=(a_n x+b_n) &P_n(x) P_{n-1}(x)+c_n P_n^2(x) 
+d_n P_{n-1}^2 (x),\end{split}\]
where $a_n$, $b_n$, $c_n$ and $d_n$ are independent of $x$;
\item[(d)] they satisfy a second order linear differential equation of the Sturm-Liouville type
\[\sigma(x)\deriv[2]{P_n}{x}+\tau(x)\deriv{P_n}{x}=\la_nP_n(x),\] where
$\sigma(x)$ is a polynomial of degree $\leq2$, $\tau(x)$ is a linear polynomial, both independent of $n$, and $\la_n$ is independent of $x$;
\item[(e)] they satisfy the differential-difference equation
\begin{equation}\pi(x)\deriv{P_n}{x}=(a_n x+b_n)P_{n}(x)+c_nP_{n-1}(x).\label{d}\end{equation}

\end{itemize}

The converse, that any polynomial set which satisfies any one of the above properties must necessarily be one of the classical orthogonal polynomial sets, also holds.
In particular, Al-Salam and Chihara \cite{ASC} showed that orthogonal polynomial sets satisfying \eqref{d} must be either Hermite, Laguerre or Jacobi polynomials.

Askey raised the more general question of what orthogonal polynomial sets have the property that their derivatives satisfy a relation of the form
$$\pi(x)\deriv{P_n}{x}=\sum_{k=n-t}^{n+s}a_{n,k}P_k(x).$$ This problem was solved by Shohat \cite{refShohat39} and later, independently, by Freud \cite{refFreud}, as well as Bonan and Nevai \cite{refBNevai}. Maroni \cite{Maroni1,Maroni2} stated the problem in a different way, trying to find all orthogonal polynomial sets whose derivatives are quasi-orthogonal, and called such orthogonal polynomial sets semi-classical.

A useful characterization of classical orthogonal polynomials is the Pearson equation
\begin{equation}\label{eq:Pearson}\deriv{}{x}[\sigma(x)w(x)]=\tau(x)w(x),\end{equation}
where $\sigma(x)$ and $\tau(x)$ are polynomials with deg$(\sigma)\leq 2$ and deg$(\tau)=1$.

Semi-classical orthogonal polynomials are defined as orthogonal polynomials for which the weight function satisfies a
Pearson equation \eqref{eq:Pearson} with deg$(\sigma)\geq 2$ or deg$(\tau)\neq 1$; see Hendriksen and van Rossum \cite{Hendriksen}.

The relationship between semi-classical orthogonal polynomials and integrable equations dates back to Shohat \cite{refShohat39} and Freud \cite{refFreud}, see also Bonan and Nevai \cite{refBNevai}. However it was not until the work of Fokas, Its and Kitaev \cite{refFIKa,refFIKb} that these equations were identified as discrete \p\ equations. The relationship between semi-classical orthogonal polynomials and the \p\ equations was discussed by Magnus \cite{refMagnus95} who, for example, showed that the coefficients in the three-term recurrence relation for the Freud weight \cite{refFreud}
\begin{equation}\ww{x}=\exp\left(-x^4+tx^2\right),\qquad x\in\R,\label{Freud}\end{equation} with $t\in\R$ a parameter,
can be expressed in terms of simultaneous solutions of the discrete equation
\begin{equation}\label{eq:dPIf}
q_n(q_{n-1}+q_n+q_{n+1})+2tq_n=n,
\end{equation}
which is {discrete \PI}\ (\dPI), as shown by Bonan and Nevai \cite{refBNevai}, and the differential equation
\begin{equation}\label{eq:PIVf}
\deriv[2]{q_n}{z}= \frac{1}{2q_n}\left(\deriv{q_n}{z}\right)^{2} + \frac{3}{2}q_n^3 + 4z q_n^2 + 2(z^2 +\tfrac12n)q_n-  \frac{n^2}{2q_n},
\end{equation}
which is a special case of the fourth \p\ equation -- see equation \eqref{eq:PIV} below -- 
with $n\in\N$. This connection between the recurrence coefficients for the Freud weight \eqref{Freud} and simultaneous solutions of \eqref{eq:dPIf} and \eqref{eq:PIVf} is due to Kitaev, see \cite{refFIKa,refFIZ}. The discrete equation \eqref{eq:dPIf} is also known as the ``Freud equation" or the ``string equation".

It is known (cf.~\cite{refIsmail}) that polynomials orthogonal with respect to exponential weights $w(x)=\exp\{-Q(x)\}$ on $\R$ for polynomial $Q(x)$ satisfy structural relations of the form
$$\deriv{P_n}{x}(x) = A_n(x) P_{n-1}(x) + B_n(x) P_n(x).$$ Such structural relations and the three-term recurrence relation reveal that the orthogonal polynomials satisfy a second order differential equation. 

It had been generally accepted that explicit expressions for the associated coefficients in the three-term recurrence relation and orthogonal polynomials were nonexistent for weights such as the Freud weight \eqref{Freud}.
To quote from the {\it{Digital Library of Mathematical Functions}} \cite[\S18.32]{DLMF}:
\begin{quote}
``A {\it{Freud weight}} is a weight function of the form
$$ w(x)=\exp\{-Q(x)\},\quad -\infty<x<\infty,$$
where $Q(x)$ is real, even, nonnegative, and continuously differentiable. Of special interest are the cases $Q(x)=x^{2m}$,  $m=1,2,\dots$. No explicit expressions for the corresponding OP's are available. However, for asymptotic approximations in terms of elementary functions for the OP's, and also for their largest zeros, see Levin and Lubinsky \cite{LevinLubinsky} and Nevai \cite{Nevai}. For a uniform asymptotic expansion in terms of Airy functions for the OP's in the case $Q(x)=x^4$ see Bo and Wong \cite{refBoWong}.
For asymptotic approximations to OP's that correspond to Freud weights with more general functions $Q(x)$ see Deift \textit{et al}.\ \cite{refDKMVZa,refDKMVZb}, Bleher and Its \cite{BleherIts}, and Kriecherbauer and McLaughlin \cite{refKrMcL}.'' \end{quote}

In \cite{refCJ}, the direct problem was studied for semi-classical Laguerre polynomials orthogonal with respect to the semi-classical Laguerre weight
\begin{equation}\label{Lw}
\ww{x} = x^{\la}\, \exp (-x^2 +tx), \qquad x\in \R^{+},\quad  \la>-1,
\end{equation} 
with $t\in\R$ a parameter, and it was shown that the coefficients in the three-term recurrence relation of these polynomials can be explicitly expressed in terms of Wronskians of parabolic cylinder functions which also arise in the description of special function solutions of the fourth \p\ equation (\PIV)
\begin{equation}\label{eq:PIV}
\deriv[2]{q}{z}= \frac{1}{2q}\left(\deriv{q}{z}\right)^{2} + \frac{3}{2}q^3 + 4z q^2 + 2(z^2 - A)q + \frac{B}{q},
\end{equation}
where $A$ and $B$ are constants, 
and the second degree, second order equation satisfied by the associated Hamiltonian function, see equation \eqref{eq:sPIV}.
\comment{i.e.\ the \PIV\ $\sigma$-equation (\sPIV)
\begin{equation}\label{eq:sPIV}
\left(\deriv[2]{\sigma}z\right)^{2} - 4\left(z\deriv{\sigma}z-\sigma\right)^{2} +4\deriv{\sigma}z\left(\deriv{\sigma}z+2\theta_0\right)
\left(\deriv{\sigma}z+2\theta_{\infty}\right)=0,
\end{equation}
with $\theta_0$ and $\theta_\infty$ constants.}

Polynomials orthogonal with respect to a symmetric measure can be generated via quadratic transformation from the classical orthogonal polynomials, cf.~\cite{refChihara}. For example, Laguerre polynomials generate a class of generalized Hermite polynomials while Jacobi polynomials give rise to a class of generalized Ultraspherical polynomials.

In this paper we are concerned with the positive, even weight function on the real line 
arising from a symmetrization of the semi-classical Laguerre weight \eqref{Lw}, namely the generalized Freud weight
\begin{equation}\ww{x}=|x|^{2\la+1}\exp\left(-x^4+tx^2\right),\qquad x\in\R\label{genFreud},\end{equation}
with parameters $\la>-1$ and $t\in\R$. 

\comment{The monic polynomials $\{S_n(x;t)\}_{n=0}^{\infty}$, 
orthogonal with respect to \eqref{genFreud} can be represented by
\begin{equation*}\int_{-\infty}^{\infty}S_n(x;t)S_m(x;t)\ww{x}\,dx=h_n\delta_{mn}, \quad h_n>0, \quad m,n\in\N.\end{equation*}}
We use two different methods to derive the differential-difference equation satisfied by generalized Freud polynomials $\{S_n(x;t)\}_{n=0}^{\infty}$, orthogonal with respect to the generalized Freud weight \eqref{genFreud}. In \S\ref{sec:scOPs} we modify the ladder operator method (cf.~\cite{refChenFeigin,refChenIsmail,refIsmail}) to derive a general formula for the coefficients of the differential-difference as well as the second order differential equation satisfied by polynomials orthogonal with respect to a generalized Freud type weight which vanishes at a point. It is important to note that the second order differential equation is linear with coefficients that are rational functions of $x$ with parameters $t$ and $\la$ involving parabolic cylinder functions.
In \S\ref{sec:scLag} we derive the generalized Freud weight \eqref{genFreud} through a symmetrization of the semi-classical Laguerre weight \eqref{Lw}.
In \S\ref{sec:genFpoly} we are concerned with specific results for generalized Freud weight \eqref{genFreud}.
We show that the coefficient $\b_n(t;\la)$ in the three-term recurrence relation
\begin{equation*}xS_n(x;t)=S_{n+1}(x;t)+\b_{n}(t;\la)S_{n-1}(x;t),\end{equation*} 
satisfied by the polynomials associated with the weight \eqref{genFreud} can be expressed in terms of Wronskians that arise in the description of special function solutions of \PIV\ which are expressed in terms of parabolic cylinder functions. 
Further we apply the results of  \S\ref{sec:scOPs} to the weight \eqref{genFreud} to derive the differential-difference equation and the linear, second order differential equations satisfied by generalized Freud polynomials $\{S_n(x;t)\}_{n=0}^{\infty}$. 
A method due to Shohat \cite{refShohat39}, based on the concept of quasi-orthogonality and applied to the weight \eqref{genFreud} is discussed in \S\ref{sho}.
\comment{We obtain explicit expressions for the coefficients of the three-term recurrence relation, differential-difference equation and second order differential equation satisfied by generalized Freud polynomials but note that the expressions are rather complicated and given in terms special function solutions of the fourth \p\ equation.}%

\section{Semi-classical orthogonal polynomials} \label{sec:scOPs}

The coefficients $A_n(x;t)$ and $B_n(x;t)$ in the relation
\begin{equation}\label{ddee}\deriv{P_n}{x}(x;t)=-B_n(x;t)P_n(x;t)+A_n(x;t)P_{n-1}(x;t),\end{equation}
satisfied by semi-classical orthogonal polynomials are of interest since differentiation of this differential-difference equation yields the second order differential equation satisfied by the orthogonal polynomials. Shohat \cite{refShohat39} gave a procedure using quasi-orthogonality to derive \eqref{ddee} for weights $\ww{x}$ such that $\displaystyle{w'(x;t)}/{\ww{x}}$ is a rational function, which we apply to \eqref{genFreud} later. This technique was rediscovered by several authors including Bonan, Freud, Mhaskar and Nevai approximately 40 years later, see \cite[p.~126--132]{Nevai} and the references therein for more detail. The method of ladder operators was introduced by Chen and Ismail in \cite{refChenIsmail}. Related work by various authors can be found in, for example, \cite{refChenIts,refChenZhang,refFvAZ,mhaskar} and a good summary of the technique is provided in \cite[Theorem 3.2.1]{refIsmail}.

In \cite{refChenFeigin}, Chen and Feigin adapt the method of ladder operators to the situation where the weight function vanishes at one point. Our next result generalizes the work in \cite{refChenFeigin}, giving a more explicit expression for the coefficients in \eqref{ddee} when the weight function is positive on the real line except for one point.

\begin{theorem}\label{Thm:ABn}
Let \begin{equation}\label{gft}\ww{x}=|x-k|^\ga \exp\{-\v(x;t)\},\qquad x,\,t,\,k\in\R,\end{equation} where $\v(x;t)$ is a continuously differentiable function on $\R$. Assume that the polynomials $\{P_n(x;t)\}_{n=0}^{\infty}$ satisfy the orthogonality relation
\[\int_{-\infty}^{\infty}P_n(x;t)P_m(x;t)\ww{x}\,dx=h_n\delta_{mn}.\]
Then, for ${\ga\geq1}$, $P_{n}(x;t)$ satisfy the differential-difference equation
\begin{equation}\label{dde}
(x-k)\,\deriv{P_n}{x}(x;t)=-B_n(x;t)P_n(x;t)+A_n(x;t)P_{n-1}(x;t),
 \end{equation}
where
\begin{subequations}\label{ABn}
\begin{align}
\label{An}
A_n(x;t)&=\frac{x-k}{h_{n-1}}\int_{-\infty}^{\infty}P_n^2(y;t)\vvxy\ww{y}\,dy+a_n(x;t),\\
B_n(x;t)&= \frac{x-k}{h_{n-1}}\int_{-\infty}^{\infty}P_n(y;t)P_{n-1}(y;t)\vvxy\ww{y}\,dy+b_n(x;t),\label{Bn}
\end{align} \end{subequations}
where
\begin{equation}\vvxy=\frac{\v'(x;t)-\v'(y;t)}{x-y},\label{def:Kxy}\end{equation}
with
\begin{subequations}\label{abn}
\begin{align}\label{Aa}a_n(x;t)&=\frac{\ga}{h_{n-1}}\int_{-\infty}^{\infty}\frac{P_n^2(y;t)}{y-k}\,\ww{y}\,dy,\\
\label{Ab}b_n(x;t)&=\frac{\ga}{h_{n-1}} \int_{-\infty}^{\infty}\frac{P_n(y;t)P_{n-1}(y;t)}{y-k}\,\ww{y}\,dy.\end{align}
\end{subequations}
\end{theorem}

\begin{proof} Consider the generalized Freud-type weight \eqref{gft}. Since $\displaystyle\deriv{P_n}{x}(x;t)$ is a polynomial of degree $n-1$ in $x$, then it can be expressed in terms of the orthogonal basis as
\begin{equation}\label{eq:orthexp}\deriv{P_n}{x}(x;t)=\sum_{j=0}^{n-1}c_{n,j}P_j(x;t).\end{equation}
Applying the orthogonality relation and integrating by parts, we obtain
\begin{align*}
c_{n,j}\,h_j&=\int_{-\infty}^{\infty}\deriv{P_n}{y}(y;t)P_j(y;t)\ww{y}\,dy\\
&=\Bigl[P_n(x;t)P_j(x;t)\ww{x}\Bigr]_{-\infty}^{\infty} 
-\int_{-\infty}^{\infty}P_n(y;t)\left\{\deriv{P_j}{y}(y;t)\ww{y}+P_j(y;t)\deriv{w}{y}(y;t)\right\}dy\\
&=-\int_{-\infty}^{\infty}P_n(y;t)P_j(y;t)\deriv{w}{y}(y;t)\,dy\\
&=\int_{-\infty}^{\infty}P_n(y;t)P_j(y;t)\left[\vvy-\frac{\ga}{y-k}\right]\ww{y}\,dy,
\end{align*}
provided that ${\ga\geq 1}$.

Now, from \eqref{eq:orthexp}, we can write
\begin{align*}
\deriv{P_n}{x}(x;t)
&=\sum_{j=0}^{n-1}\frac{1}{h_j}\left\{\int_{-\infty}^{\infty}P_n(y;t)P_j(y;t)\left[\vvy-\frac{\ga}{y-k}\right]\ww{y}\,dy\right\}P_j(x;t)\\
&=\int_{-\infty}^{\infty}P_n(y;t)\sum_{j=0}^{n-1}\frac{P_j(y;t)P_j(x;t)}{h_j}\left[\vvy-\frac{\ga}{y-k}\right]\ww{y}\,dy\\
&=\int_{-\infty}^{\infty}P_n(y;t)\sum_{j=0}^{n-1}\frac{P_j(y;t)P_j(x;t)}{h_j}\left[\vvy-\vvx\right]\ww{y}\,dy\\
&\qquad\qquad+\vvx\int_{-\infty}^{\infty}P_n(y;t)\sum_{j=0}^{n-1}\frac{P_j(y;t)P_j(x;t)}{h_j}\,\ww{y}\,dy\\
&\qquad\qquad-\ga\int_{-\infty}^{\infty}\frac{P_n(y;t)}{y-k}\sum_{j=0}^{n-1}\frac{P_j(y;t)P_j(x;t)}{h_j}\,\ww{y}\,dy
\\
&=\int_{-\infty}^{\infty}P_n(y;t)\sum_{j=0}^{n-1}\frac{P_j(y;t)P_j(x;t)}{h_j}\left[\vvy-\vvx\right]\ww{y}\,dy\\
&\qquad\qquad-\ga\int_{-\infty}^{\infty}\frac{P_n(y;t)}{y-k}\sum_{j=0}^{n-1}\frac{P_j(y;t)P_j(x;t)}{h_j}\,\ww{y}\,dy.
\end{align*}
Next, using the orthogonality relation again, we obtain
\begin{align*}
{(x-k) \deriv{P_n}{x}(x;t)}
&=(x-k) \int_{-\infty}^{\infty}P_n(y;t)\sum_{j=0}^{n-1}\frac{P_j(y;t)P_j(x;t)}{h_j}\left[\vvy-\vvx\right]\ww{y}\,dy\\
&\qquad\qquad-\ga\int_{-\infty}^{\infty}P_n(y;t)\left(\frac{x-k}{y-k}\right)\sum_{j=0}^{n-1}\frac{P_j(y;t)P_j(x;t)}{h_j}\,\ww{y}\,dy.\\
&=(x-k) \int_{-\infty}^{\infty}P_n(y;t)\sum_{j=0}^{n-1}\frac{P_j(y;t)P_j(x;t)}{h_j}\left[\vvy-\vvx\right]\ww{y}\,dy\\
&\qquad\qquad-\ga\int_{-\infty}^{\infty}P_n(y;t)\frac{x-y}{y-k}\sum_{j=0}^{n-1}\frac{P_j(y;t)P_j(x;t)}{h_j}\,\ww{y}\,dy.
\end{align*}
It now follows from the Christoffel-Darboux formula (cf.~\cite{refIsmail})
$$ \sum_{j=0}^{n-1}\frac{P_j(y;t)P_j(x;t)}{h_j}=\frac{P_n(x;t)P_{n-1}(y;t)-P_n(y;t)P_{n-1}(x;t)}{(x-y)h_{n-1}},$$ that
\begin{align*}
A_n(x;t)&=\frac{x-k}{h_{n-1}}\int_{-\infty}^{\infty}P_n^2(y;t)\vvxy\ww{y}\,dy+a_n(x;t)\\
B_n(x;t)&= \frac{x-k}{h_{n-1}}\int_{-\infty}^{\infty}P_n(y;t)P_{n-1}(y;t)\vvxy\ww{y}\,dy +b_n(x;t),
\end{align*} with
 $a_n(x;t)$ and $b_n(x;t)$ given by \eqref{abn}.
\end{proof}
\begin{lemma}\label{lemmaeven}
Consider the weight defined by \eqref{gft} and assume that $\v(x;t)$ is an even, continuously differentiable function on $\R$. Assume that the polynomials $\{P_n(x;t)\}_{n=0}^{\infty}$ satisfy the orthogonality relation
\[\int_{-\infty}^{\infty}P_n(x;t)P_m(x;t)\ww{x}\,dx=h_n\delta_{mn}.\]
 and the three-term recurrence relation
\begin{equation}\label{3trr}
P_{n+1}(x;t)=xP_{n}(x;t)-\b_n(t;\la)P_{n-1}(x;t),
\end{equation}
with $P_0=1$ and $P_1=x$. Then the polynomials $P_n(x;t)$ satisfy
\begin{align}  \int_{-\infty}^{\infty}\frac{P_n^2(y;t)}{y-k}\,\ww{y}\,dy&=0,\\
\int_{-\infty}^\infty  \frac{P_n(y;t)P_{n-1}(y;t)}{y-k}\,\ww{y}\,dy&= \tfrac12[1-(-1)^n]\,h_{n-1},
\label{int22}
\end{align} where $n\in\N$ and
\[ h_n = \int_{-\infty}^\infty  {P_n^2(y;t)\ww{y}}\,dy.\]\end{lemma}

\begin{proof} Since $\ww{x}$ is even when $\v(x;t)$ is assumed to be even, the integrand in \eqref{Aa} is odd and hence $a_n(x;t)=0$.

Furthermore, the monic orthogonal polynomials $P_{n}(x;t)$ satisfy the three-term recurrence relation \eqref{3trr},
with $P_0=1$ and $P_1=x$, hence
\begin{align*}
\int_{-\infty}^\infty \frac{P_n(y;t)P_{n-1}(y;t)}{y-k}\,\ww{y}\,dy
&=\int_{-\infty}^\infty  \frac{\big[ yP_{n-1}(y;t)-\b_{n-1}P_{n-2}(y;t)\big]P_{n-1}(y;t)}{y-k}\,\ww{y} \,dy\\
&=\int_{-\infty}^\infty  P_{n-1}^2(y;t)\ww{y}\,dy +k\int_{-\infty}^{\infty}\frac{P_{n-1}^2(y;t)}{y-k}\, \ww{y}\,dy\\
&\qquad- \b_{n-1}\int_{-\infty}^\infty  \frac{P_{n-1}(y;t)P_{n-2}(y;t)}{y-k}\,\ww{y}\,dy\\
&= h_{n-1} - \b_{n-1}\int_{-\infty}^\infty  \frac{P_{n-1}(y;t)P_{n-2}(y;t)}{y-k}\,\ww{y}\,dy,
\end{align*} since the integrand in the second integral is odd.
Hence, if we define
\[ I_n=\int_{-\infty}^\infty  \frac{P_n(y;t)P_{n-1}(y;t)}{y-k}\,\ww{y}\,dy,\]
then $I_n$ satisfies the recurrence relation
\begin{equation*} I_n = h_{n-1} - \b_{n-1}I_{n-1}= h_{n-1} -\frac{h_{n-1}}{h_{n-2}} I_{n-1},
\end{equation*}
since $\b_n=h_{n}/h_{n-1}$. Iterating this gives
\begin{align*}
I_n &= \frac{h_{n-1}}{h_{n-3}} I_{n-2} 
= h_{n-1}-\frac{h_{n-1}}{h_{n-4}} I_{n-3} 
= \frac{h_{n-1}}{h_{n-5}} I_{n-4} 
= h_{n-1}-\frac{h_{n-1}}{h_{n-6}} I_{n-5}, 
\end{align*} and so on.
Hence, by induction,
\begin{align*}
I_{2N} =\frac{h_{2N-1}}{h_{1}} I_{2},\qquad I_{2N+1} =h_{2N}-\frac{h_{2N}}{h_{1}} I_{2},
\end{align*}
and so, since
\begin{align*}
I_2&=\int_{-\infty}^\infty  \frac{P_2(y;t)P_{1}(y;t)}{y-k}\,\ww{y}\,dy\\
& =\int_{-\infty}^\infty  P_2(y;t)\ww{y}\,dy+k\int_{-\infty}^{\infty}\frac{P_2(y;t)}{y-k}\,\ww{y}\,dy=0,\end{align*}
we have that \[ I_{2N} =0,\qquad I_{2N+1} =h_{2N},\]
as required.
\end{proof}

\begin{corollary}\label{cor:ABn}
Let \begin{equation*}\ww{x}=|x-k|^\ga \exp\{-\v(x;t)\},\qquad x,\,t,\,k\in\R,\end{equation*} where $\v(x;t)$ is an even, continuously differentiable function on $\R$. Assume that the polynomials $\{P_n(x;t)\}_{n=0}^{\infty}$ satisfy the orthogonality relation
\[\int_{-\infty}^{\infty}P_n(x;t)P_m(x;t)\ww{x}\,dx=h_n\delta_{mn}.\]
Then, for ${\ga\geq1}$, $P_{n}(x;t)$ satisfy the differential-difference equation
\begin{equation*}
(x-k)\,\deriv{P_n}{x}(x;t)=-B_n(x;t)P_n(x;t)+A_n(x;t)P_{n-1}(x;t),
 \end{equation*}
where
\begin{subequations}\label{ABn2}
\begin{align}
A_n(x;t)&=\frac{x-k}{h_{n-1}}\int_{-\infty}^{\infty}P_n^2(y;t)\vvxy\ww{y}\,dy,\\
B_n(x;t)&= \frac{x-k}{h_{n-1}}\int_{-\infty}^{\infty}P_n(y;t)P_{n-1}(y;t)\vvxy\ww{y}\,dy+\tfrac12{\ga}[1-(-1)^n].
\end{align}
\end{subequations}
\end{corollary}
\begin{proof} The result is an immediate consequence of Theorem \ref{Thm:ABn} and Lemma \ref{lemmaeven}.
\end{proof}
\begin{lemma}\label{lem22} Let $\ww{x}$ be the weight defined by \eqref{gft} with $\v(x;t)$ an even, continuously differentiable function on $\R$ and $A_n(x;t)$ and $B_n(x;t)$ defined by \eqref{ABn2}. Then
\begin{equation*}B_n(x;t)+B_{n+1}(x;t)=\frac{x A_n(x;t)}{\b_n}+\ga-(x-k)\vvx,\end{equation*}
when $\ga\geq1$.
\end{lemma}
\begin{proof}From \eqref{ABn2}, \eqref{3trr} and the fact that $h_n=h_{n-1}\b_n$  we have
\begin{align}\nonumber
B_n(x;t)+B_{n+1}(x;t)&=\frac{x-k}{h_{n-1}}\int_{-\infty}^{\infty}P_n(y;t)P_{n-1}(y;t)\vvxy\ww{y}\,dy\nonumber\\
&\qquad+\frac{x-k}{h_{n}}\int_{-\infty}^{\infty}P_n(y;t)P_{n+1}(y;t)\vvxy\ww{y}\,dy\nonumber\\
&\qquad+\tfrac12{\ga}[1-(-1)^{n}]+\tfrac12{\ga}[1-(-1)^{n+1}]\nonumber\\
&=\frac{x-k}{h_{n}}\left\{\int_{-\infty}^{\infty}P_{n}(y;t)\left[\b_nP_{n-1}(y;t)+P_{n+1}(y;t)\right]
\vvxy\ww{y}\,dy\right\}+\ga\nonumber\\
&=\frac{x-k}{h_{n}}\int_{-\infty}^{\infty}yP^2_{n}(y;t)\vvxy\ww{y}\,dy+\ga\nonumber\\
&=\frac{x-k}{h_{n}}\int_{-\infty}^{\infty}P_n^2(y;t)\left[\vvy-\vvx\right]\ww{y}\,dy\nonumber\\
&\qquad+\frac{x(x-k)}{h_n}\int_{-\infty}^{\infty}P_n^2(y;t)\vvxy\ww{y}\,dy+\ga\nonumber\\
&=\frac{x-k}{h_n}\int_{-\infty}^{\infty}P_n^2(y;t)\vvy\ww{y}\,dy\nonumber\\
&\qquad-\frac{x-k}{h_n}\vvx\int_{-\infty}^{\infty}P_n^2(y;t)\ww{y}\,dy+\frac{x}{h_n}h_{n-1}A_n(x;t)+\ga\nonumber\\
&=\frac{\gamma(x-k)}{h_n}\int_{-\infty}^{\infty}P_n^2(y;t)\frac{\ww{y}}{y-k}\,dy-\frac{x-k}{h_n}\int_{-\infty}^{\infty}P_n^2(y;t)\deriv{w}{y}(y;t)\,dy\nonumber\\
&\qquad-(x-k)\vvx+\frac{xA_n(x;t)}{\b_n}+\ga,\label{bet}\end{align}
since \[\ds{\deriv{w}{y}(y;t)=\left[-\vvy+\frac{\ga}{y-k}\right]\ww{y}}.\] The first integral in \eqref{bet} vanishes since the integrand is odd, hence it follows, using integration by parts, that
\begin{align*}
B_n(x;t)+B_{n+1}(x;t)&=-\frac{x-k}{h_n}\Big\{\Bigr[P_n^2(x;t)\ww{x}\Bigr]_{-\infty}^{\infty}-\int_{-\infty}^{\infty}2P_n(y;t)\deriv{P_n}{y}(y;t)\ww{y}\,dy\Bigr\}\\
&\qquad -(x-k)\vvx+\frac{x A_n(x;t)}{\b_n}+\ga,
\end{align*}
and the result follows from the orthogonality of $P_n$.
\end{proof}

\begin{lemma} Let $\ww{x}$ be the weight defined by \eqref{gft} with $\v(x;t)$ an even, continuously differentiable function on $\R$ and $A_n(x;t)$ and $B_n(x;t)$ defined by \eqref{ABn2}. Then
\begin{align} \left\{-(x-k)\deriv{}{x}+B_n(x;t)+(x-k)\vvx-\ga\right\}&P_{n-1}(x;t)\label{dde2} 
=\frac{A_{n-1}(x;t)}{\b_{n-1}}P_n(x;t).\end{align}
\end{lemma}

\begin{proof}
From the differential-difference equation \eqref{dde} we have
\begin{align*}
(x-k)\deriv{P_{n-1}}{x}(x;t)&=-B_{n-1}(x;t)P_{n-1}(x;t)+A_{n-1}(x;t)P_{n-2}(x;t)\\
&=-B_{n-1}(x;t)P_{n-1}(x;t) 
+\frac{A_{n-1}(x;t)}{\b_{n-1}}\left[xP_{n-1}(x;t)-P_n(x;t)\right],\end{align*}
using the recurrence relation \eqref{3trr}. Hence, using Lemma \ref{lem22},
\begin{align*} \frac{A_{n-1}(x;t)}{\b_{n-1}}P_n(x;t) 
&=-(x-k)\deriv{P_{n-1}}{x}(x;t)-B_{n-1}(x;t)P_{n-1}(x;t) 
+\frac{x A_{n-1}(x;t)}{\b_{n-1}}P_{n-1}(x;t)\\
&=-(x-k)\deriv{P_{n-1}}{x}(x;t)-B_{n-1}(x;t)P_{n-1}(x;t)\\ &\qquad
+\left\{B_n(x;t)+B_{n-1}(x;t)-\ga+(x-k)\vvx\right\}P_{n-1}(x;t)\\
&=-(x-k)\deriv{P_{n-1}}{x}(x;t)+B_{n}(x;t)P_{n-1}(x;t) 
+\left[(x-k)\vvx-\ga\right]P_{n-1}(x;t)\\
&=\left\{-(x-k)\deriv{}{x}+B_n(x;t)+(x-k)\vvx-\ga\right\}P_{n-1}(x;t).
\end{align*}
\end{proof}

\begin{theorem} \label{thm:gende}Let  
$\ww{x}=|x-k|^\ga \exp\{-\v(x;t)\}$, for  $x,\,t,\,k\in\R$,
 with $\v(x;t)$ an even, continuously differentiable function on $\R$. Then
\begin{subequations}\begin{equation}\label{eq:gende}
(x-k)\deriv[2]{P_n}{x}(x;t)+R_n(x;t)\deriv{P_n}{x}(x;t)+T_n(x;t)P_n(x;t)=0,
\end{equation}
where
\begin{align}\label{coef1}
R_n(x;t)&=\ga-(x-k)\vvx-\frac{x-k}{A_n(x;t)}\,\deriv{A_n}{x}(x;t)+1,\\[5pt]
\label{coef2}T_n(x;t)&=\frac{A_n(x;t)A_{n-1}(x;t)}{(x-k)\b_{n-1}}+\deriv{B_n}{x}(x;t)+\frac{\ga B_n(x;t)}{x-k}\nonumber\\
&\qquad-B_n(x;t)\left[\vvx +\frac{B_n(x;t)}{x-k}\right]-\frac{B_n(x;t)}{A_n(x)}\deriv{A_n}{x}(x;t),
\end{align}
with
\begin{align*}
A_n(x;t)&=\frac{x-k}{h_{n-1}}\int_{-\infty}^{\infty}P_n^2(y;t)\vvxy\ww{y}\,dy\\
B_n(x;t)&= \frac{x-k}{h_{n-1}}\int_{-\infty}^{\infty}P_n(y;t)P_{n-1}(y;t)\vvxy\ww{y}\,dy 
+\tfrac12{\ga}[1-(-1)^n].
\end{align*}\end{subequations}
\end{theorem}
\begin{proof}
Differentiating both sides of \eqref{dde} with respect to $x$ we obtain
\begin{align}\label{eq:dif1}
(x-k)\deriv[2]{P_n}{x}(x;t)&=-[B_n(x;t)+1]\deriv{P_n}{x}x,t)+\deriv{A_n}{x}(x;t)P_{n-1}(x;t)\nonumber\\
&\qquad-\deriv{B_n}{x}(x;t)P_n(x;t)+A_n(x;t)P_{n-1}(x;t).\end{align}
Substituting \eqref{dde2} into \eqref{eq:dif1} yields
\begin{align}
(x-k)\deriv[2]{P_n}{x}(x;t)
&=\nonumber-[B_n(x;t)+1]\deriv{P_n}{x}(x;t)-\left[\deriv{B_n}{x}(x;t)+\frac{A_n(x;t)A_{n-1}(x;t)}{\b_{n-1}(x-k)}\right]P_n(x;t)\\&\qquad+\left\{\deriv{A_n}{x}(x;t)+A_n(x;t)\left[\frac{B_n(x;t)}{x-k}+\vvx-\frac{\ga}{x-k}\right]\right\}P_{n-1}(x;t),\label{eq:fin}\end{align}
and the results follows by substituting $P_{n-1}(x;t)$ in \eqref{eq:fin} using \eqref{dde}.
\end{proof}

\section{Semi-classical Laguerre weight}\label{sec:scLag}
We recall that the semi-classical Laguerre weight \cite{refCJ,Hendriksen} is given by
\begin{align}\label{laguerreweight}
\ww{x} = x^{\la}\, \exp (-x^2 +tx), \quad x\in \R^{+},
\end{align}
where $\la > -1$. 
The weight function satisfies the differential equation $$x\deriv{w}{x}(x;t) + (2x^2-tx-\la) \ww{x} = 0,$$ which is the Pearson equation \eqref{eq:Pearson} 
with $\sigma(x) =x$ and $ \tau(x;t) = -2x^2+tx+\la +1$.

Explicit expressions for the recurrence coefficients $\a_n(t;\la)$ and $\b_n(t;\la)$ in the three-term recurrence relation
\begin{equation}xP_n(x;t)=P_{n+1}(x;t)+\a_n(t;\la)P_n(x;t)+\b_n(t;\la)P_{n-1}(x;t),\label{eq:scLag2}\end{equation}
associated with the semi-classical Laguerre weight \eqref{laguerreweight} were obtained in \cite{refCJ}, and are given in the following theorem.
\begin{theorem}{Suppose $\Psi_{n,\la}(z)$ is given by
\begin{equation*}\label{eq:PT.SF.eq411}
\Psi_{n,\la}(z)=\W\left(\psi_{\la},\deriv{\psi_{\la}}z,\ldots,\deriv[n-1]{\psi_{\la}}z\right),\qquad \Psi_{0,\la}(z)=1,
\end{equation*}
where
\begin{equation*}\label{eq:PT.SF.eq413}
\psi_{\la}(z)=\begin{cases} D_{-\la-1}\big(-\sqrt2\,z\big)\exp\big(\tfrac12z^2\big),\quad
&\mbox{\rm if}\quad \la\not\in\N,\\[2.5pt]
\ds \deriv[m]{}{z}\left\{\big[1+\erf(z)\big]\exp(z^2)\right\}, &\mbox{\rm if}\quad \la=m\in\N,
\end{cases}\end{equation*}
with $D_{\nu}(\zeta)$ is the {parabolic cylinder function} and $\erfc(z)$ the complementary error function
\begin{equation}\label{def:erfc}
\erfc (z)=\frac{2}{\sqrt{\pi}} \int_z^\infty \exp (-t^2) \,dt,
\end{equation}
Then coefficients $\a_n(t;\la)$ and $\b_n(t;\la)$ in the recurrence relation \eqref{eq:scLag2}
associated with the semi-classical Laguerre weight \eqref{laguerreweight} are given by
\begin{subequations}\label{eq:scLag31ab}\begin{align}\label{eq:scLag31a}
\a_n(t;\la)&=\tfrac12q_n(z)+\tfrac12t,\\ \label{eq:scLag31b}
\b_n(t;\la)&=-\tfrac18\deriv{q_n}z-\tfrac18 q_n^2(z) -\tfrac14zq_n(z)+\tfrac14\la+\tfrac12n,\end{align}\end{subequations}
with $z=\tfrac12t$, where
\begin{equation*}\label{eq:PT.SF.eq410}q_n(z)=-2z+\deriv{}{z}\ln\frac{\Psi_{n+1,\la}(z)}{\Psi_{n,\la}(z)},\end{equation*}
which satisfies \PIV\ \eqref{eq:PIV} with parameters $(A,B)=(2n+\la+1,-2\la^2)$.
}\end{theorem}

\begin{proof}See \cite{refCJ}. \end{proof}

\subsection{Symmetrization of semi-classical Laguerre weight}
In this section we show that symmetrizing the semi-classical Laguerre weight \eqref{laguerreweight} gives rise to the generalized Freud weight \eqref{genFreud}.

Let $\lbrace L^{(\la)}_n(x;t)\rbrace$ denote the monic semi-classical Laguerre polynomials orthogonal with respect to
semi-classical Laguerre weight \eqref{laguerreweight} with \begin{align*}
 \int_0^{\infty} L^{(\la)}_m(x;t)  L^{(\la)}_n(x;t)
 \, x^{\la}\exp(-x^2+tx) \,d{x} = K_n(t)\, \delta_{mn}.
\end{align*}

Define
\begin{align}\label{quadratic}
S_{2m}(x;t) =L^{(\la)}_m(x^2;t); \quad S_{2m+1}(x;t) =xQ^{(\la)}_n(x^2;t),
\end{align}
where \begin{align}\label{Kernelsym}
Q^{(\la)}_n(x;t)= \frac{1}{x}\left[ L^{(\la)}_{n+1}(x;t) - \dfrac{L^{(\la)}_{n+1}({0;t})}{L^{(\la)}_n({0;t})} L^{(\la)}_n(x;t) \right],
\end{align}
are also monic and of degree $n$, for $x\neq0$.

 Then
\begin{align*}
 \int_0^{\infty} L^{(\la)}_m(x;t) L^{(\la)}_n(x;t) x^{\la}\exp(-x^2+tx) \,d{x} 
 & = 2\int_0^{\infty} L^{(\la)}_m(x^2;t)  L^{(\la)}_n(x^2;t) x^{2\la+1} \exp\left(-x^4+tx^2\right) d{x}
 \\& = \int_{-\infty}^{\infty} L^{(\la)}_m(x^2;t)  L^{(\la)}_n(x^2;t) |x|^{2\la+1}  \exp\left(-x^4+tx^2\right) d{x} 
 \\&= \int_{-\infty}^{\infty}S_{2m}(x;t)  S_{2n}(x;t)  |x|^{2\la+1}\exp\left(-x^4+tx^2\right) d{x}\\&=K_n(t) \delta_{mn}.
\end{align*}
which implies that
 $\lbrace S_{2m}(x;t) \rbrace_{m=0}^{\infty}$ is a symmetric orthogonal sequence with respect to the even weight
 \begin{equation*} \ww{x} =|x|^{2\la+1}\exp (-x^4 +tx^2),\end{equation*} on $\R$, i.e.\ the generalized Freud weight \eqref{genFreud}.

 It is proved in \cite[Thm 7.1]{refChihara} that the kernel polynomials $Q_{m}^L(x;t)$
are orthogonal with respect to $x \ww{x}= x^{\la+1}\exp(-x^2+tx)$. Hence
\begin{align*}K_n(t) \delta_{mn}&=
\int_0^{\infty} Q^{(\la)}_m(x;t)  Q^{(\la)}_n(x;t)  x^{\la+1}\exp(-x^2+tx) \,d{x} \\&=
 2\int_0^{\infty} Q^{(\la)}_m(x^2;t)  Q^{(\la)}_n(x^2;t)  x^{2\la+3}\exp\left(-x^4+tx^2\right) d{x}
 \\&=
 \int_{-\infty}^{\infty} \left[ xQ^{(\la)}_m(x^2;t)\right)]\left[ xQ^{(\la)}_n(x^2;t)\right] |x|^{2\la+1}\exp\left(-x^4+tx^2\right) d{x}
 \nonumber\\&=
 \int_{-\infty}^{\infty} S_{2m+1}(x;t) S_{2n+1}(x;t)  |x|^{2\la+1}\exp\left(-x^4+tx^2\right)d{x}.
\end{align*}
Lastly, since in each case the integrand is odd, we have that
\begin{align*}
\int_{-\infty}^{\infty} S_{2m+1}(x;t) S_{2n}(x;t)   |x|^{2\la+1}\exp\left(-x^4+tx^2\right) d{x}=0,\qquad m,n\in\N,\end{align*} 
and so we conclude that $\lbrace S_{n}(x;t) \rbrace_{n=0}^{\infty}$ is a sequence of polynomials orthogonal with respect to the even weight \eqref{genFreud} on $\R.$

\comment{Note that
\begin{align*}\label{dual}
 \widetilde{\w}(x;t) = |x|^{-1} w(x^2) = |x|^{2\la -1} \exp(-x^4 +tx^2).
 \end{align*} is another symmetric dual weight function for the semi-classical Laguerre weight (cf.~\cite{Masjed}).}

\section{The generalized Freud weight}\label{sec:genFpoly}
\subsection{The moments of the generalized Freud weight}
It is shown in \cite{refCJ} that the moments of the semi-classical weight provide the link between the weight and
the associated \p\ equation. The recurrence coefficients in the three-term recurrence relations associated with semi-classical orthogonal polynomials can often be expressed in terms of solutions of the \p\ equations and associated discrete \p\ equations.

For the generalized Freud weight \eqref{genFreud},
the first moment, $\mu_0(t;\la)$, can be obtained using the integral representation of a parabolic cylinder (Hermite-Weber) function $\WhitD{v}(\xi)$. By definition
\[\begin{split}
\mu_0(t;\la)&=\int_{-\infty}^{\infty}|x|^{2\la+1}\exp\left(-x^4+tx^2\right)dx\\
&=2\int_{0}^{\infty}x^{2\la+1}\exp\left(-x^4+tx^2\right)dx\\
&=\int_{0}^{\infty}y^{\la}\exp\left(-y^2+ty\right)dy\\
&= \frac{\Gamma(\la+1)}{2^{(\la+1)/2}}\,\exp\left(\tfrac18t^2\right)\WhitD{-\la-1}\big(-\tfrac12\sqrt2\,t\big).
\end{split}\]
since the parabolic cylinder function $\WhitD{v}(\xi)$ has the integral representation \cite[\S12.5(i)]{DLMF}
\[\WhitD{v}(\xi) = \dfrac{\exp(-\tfrac{1}{4}\xi^2)}{\Gamma(-v)} \int_{0}^{\infty} s^{-v-1} \, \exp\left( -\tfrac{1}{2}s^2 -\xi s \right) d{s},\qquad \Re(\nu)<0.
\]
The even moments are
\begin{align*}
\mu_{2n}(t;\la) &= \int_{-\infty}^{\infty} x^{2n}\, |x|^{2\la+1}\exp\left(-x^4+tx^2\right) d{x} \\&=
 \deriv[n]{}{t}\left(\int_{-\infty}^{\infty}|x|^{2\la+1}\exp\left(-x^4+tx^2\right) d{x} \right),\nonumber\\&
 = \deriv[n]{}{t}\mu_0(t;\la), \quad n\geq1,
\end{align*}
whilst the odd ones are
\begin{align*}
\mu_{2n+1}(t;\la) &= \int_{-\infty}^{\infty} x^{2n+1}\, |x|^{2\la+1}\exp\left(-x^4+tx^2\right) d{x} =0, \quad n\in\N,
\end{align*}
since the integrand is odd.

We note that
\comment{\[\begin{split}\deriv{}{t}\mu_0(t;\la) &= \frac{\Gamma(\la+1)}{2^{(\la+1)/2}}\,\exp\left(\tfrac18t^2\right)\left\{\tfrac12t\WhitD{-\la-1}\big(-\tfrac12\sqrt2\,t\big)
\right. \\ &\qquad\qquad+\left. \tfrac12\sqrt{2}\,\WhitD{-\la}\big(-\tfrac12\sqrt2\,t\big)\right\} \\
&= \frac{\Gamma(\la+2)}{2^{(\la+2)/2}}\,\exp\left(\tfrac18t^2\right)\WhitD{-\la-2}\big(-\tfrac12\sqrt2\,t\big)\\
& =\mu_0(t;\la+1),
\end{split}\]
since
\[\begin{split}
&\deriv{}{z}\WhitD{\nu}(z)=\tfrac12z\WhitD{\nu}(z)-\WhitD{\nu+1}(z),\\
&\WhitD{\nu+1}(z)-z\WhitD{\nu}(z)+\nu\WhitD{\nu-1}(z)=0.\end{split}\]}%
\begin{align*}
\mu_{2n}(t;\la) &= \int_{-\infty}^{\infty} x^{2n}\, |x|^{2\la+1}\exp\left(-x^4+tx^2\right) d{x}\\
&= \int_{-\infty}^{\infty} |x|^{2\la+2n+1}\exp\left(-x^4+tx^2\right) d{x}\\
&= \mu_0(t;\la+n).
\end{align*}
Also, when $\la=n\in\N$, then
\[\WhitD{-n-1}\big(-\tfrac12\sqrt2\,t\big)=\tfrac12\sqrt{2\pi}\,\deriv[n]{}{t}\left\{\left[1+\erf\left(\tfrac12t\right)\right]\exp\left(\tfrac18t^2\right)\right\},\]
where $\erf(z)$ is the error function \cite[\S12.7(ii)]{DLMF}.

\subsection{Recurrence coefficients of the generalized Freud polynomials}
Monic orthogonal polynomials with respect to the symmetric generalized Freud weight \eqref{genFreud} satisfy the three-term recurrence relation
\begin{equation}xS_n(x;t)=S_{n+1}(x;t)+\b_{n}(t;\la)S_{n-1}(x;t),\label{eq:gfrr}\end{equation} 
where $S_{-1}(x;t)=0$ and $S_0(x;t)=1.$

Our interest is to determine explicit expressions for the recurrence coefficients $\b_n(t;\la)$ in the three-term recurrence relation \eqref{eq:gfrr}. First we relate them to solutions of \PIV\ \eqref{eq:PIV}.

\begin{lemma}{The recurrence coefficients $\b_n(t;\la)$ in \eqref{eq:gfrr} satisfy the equation
\begin{equation}\label{eq:bn}
\deriv[2]{\b_n}{t}=\frac{1}{2\b_n}\left(\deriv{\b_n}{t}\right)^2+\tfrac32\b_n^3-t\b_n^2+(\tfrac18 t^2-\tfrac12A_n)\b_n+\frac{B_n}{16\b_n},\end{equation}
where the parameters $A_n$ and $B_n$ are given by
\begin{equation*}\begin{array}{l@{\quad}l}A_{2n} = -2\la-n-1, & A_{2n+1}=\la-n,\\
B_{2n} = -2n^2, & B_{2n+1}=-2(\la+n+1)^2.\end{array}\end{equation*}
Further $\b_n(t;\la)$ satisfies the nonlinear difference equation
\begin{equation}\label{eq:dPI}
\b_{n+1}+\b_n+\b_{n-1}=\tfrac12t+\frac{2n+(2\la+1)[1-(-1)^n]}{8\b_n},\end{equation}
which is known as discrete \PI\ (\dPI).
}\end{lemma}

\begin{proof}See, for example, \cite[Theorem 6.1]{refFvAZ}.\end{proof}

\begin{remarks}{\rm\qquad\phantom{x}
\begin{enumerate}
\item Equation \eqref{eq:bn} is equivalent to \PIV\ \eqref{eq:PIV} through the transformation $\b_n(t;\la)=\tfrac12 w(z)$, with $z=-\tfrac12t$. Hence
\begin{subequations}\label{bnp4} \begin{align}
\b_{2n}(t;\la)&=\tfrac12 w\big(z;-2\la-n-1,-2n^2\big),\\
\b_{2n+1}(t;\la)&=\tfrac12 w\big(z;\la-n,-2(\la+n+1)^2\big),
\end{align}\end{subequations}
with $z=-\tfrac12t$,
where $w(z;A,B)$ satisfies \PIV\ \eqref{eq:PIV}. The conditions on the \PIV\ parameters in \eqref{bnp4} are precisely those for which \PIV\ has \eqref{eq:PIV} has solutions expressible in terms of the parabolic cylinder function
\[\psi(z)=\mu_0(-2z;\la)=\frac{\Gamma(\la+1)}{2^{(\la+1)/2}}\,\exp\left(\tfrac12z^2\right)\WhitD{-\la-1}\big(\sqrt2\,z\big),\]
\cite{refFW01,refOkamotoPIIPIV}; see also \cite[Theorem 3.5]{refCJ}.

\item The link between the differential equation \eqref{eq:bn} and the difference equation \eqref{eq:dPI} is given by the B\"{a}cklund transformations
\begin{subequations}
\begin{align}
\b_{n+1}&=\frac{1}{2\b_n}\deriv{\b_n}{t}-\tfrac12\b_n+\tfrac14t+\frac{\ga_n}{4\b_n},\label{def:bnp}\\
\b_{n-1}&=-\frac{1}{2\b_n}\deriv{\b_n}{t}-\tfrac12\b_n+\tfrac14t+\frac{\ga_n}{4\b_n},\label{def:bnm}
\end{align}\end{subequations}
with $\ga_n=\tfrac12n+\tfrac14(2\la+1)[1-(-1)^n]$.
Letting $n\to n+1$ in \eqref{def:bnm} and substituting \eqref{def:bnp} gives the differential equation \eqref{eq:bn}, whilst
eliminating the derivative yields the difference equation \eqref{eq:dPI}.
\end{enumerate}
}\end{remarks}

\begin{lemma}{The recurrence coefficients $\b_n(t;\la)$ in \eqref{eq:gfrr} are given by
\begin{subequations}
\begin{align}
\b_{2n}(t;\la) &=\deriv{}{t}\ln\frac{\tau_n(t;\la+1)}{\tau_n(t;\la)},\\
\b_{2n+1}(t;\la) &=\deriv{}{t}\ln\frac{\tau_{n+1}(t;\la)}{\tau_n(t;\la+1)},
\end{align}\end{subequations}
where $\tau_n(t;\la)$ is the Wronskian given by
\begin{align}\label{taun}
\tau_n(t;\la)&=\mathcal{W}\left(\phi_{\la},\deriv{\phi_{\la}}{t},\ldots,\deriv[n-1]{\phi_{\la}}{t}\right),
\end{align}
with \begin{equation*}
\phi_{\la}(t)=\mu_0(t;\la)=\frac{\Gamma(\la+1)}{2^{(\la+1)/2}}\,\exp\left(\tfrac18t^2\right)\WhitD{-\la-1}\big(-\tfrac12\sqrt2\,t\big).
\end{equation*}
}\end{lemma}

\begin{proof}{From the parabolic cylinder solutions of \PIV\ \eqref{eq:PIV} given in \cite[Theorem 3.5]{refCJ}, it is easily shown that the equation
\begin{equation}\label{eqb}
\deriv[2]{y}{t}=\frac{1}{2y}\left(\deriv{y}{t}\right)^2+\tfrac32y^3-ty^2+(\tfrac18 t^2-\tfrac12A)y+\frac{B}{16y},\end{equation}
has the solutions $\left\{y_{n}^{[j]}\big(t;A_{n}^{[j]},B_{n}^{[j]}\big)\right\}_{j=1,2,3}$, 
\begin{subequations}\label{eqbsol}
\begin{align}
&y_{n}^{[1]}\big(t;\la+2n-1, -2\la^2\big)= \tfrac12t+\deriv{}{t}\ln\frac{\tau_{n-1}(t;\la)}{\tau_{n}(t;\la)},\\
&y_{n}^{[2]}\big(t;-2\la-n-1, -2n^2\big) =\deriv{}{t}\ln\frac{\tau_n(t;\la+1)}{\tau_n(t;\la)},\\
&y_{n}^{[3]}\big(t;\la-n+1, -2(\la+n)^2\big) =\deriv{}{t}\ln\frac{\tau_{n}(t;\la)}{\tau_{n-1}(t;\la+1)},
\end{align}\end{subequations}
where $\tau_n(t;\la)$ is the Wronskian \eqref{taun}. Comparing \eqref{bnp4} and \eqref{eqbsol} gives the desired result.
}\end{proof}

The first few recurrence coefficients $\b_n(t;\la)$ are given by
\[\begin{split}
\b_1(t;\la)&=\Ph,\\
\b_2(t;\la)&=-\frac{2\Ph^2-t\Ph-\la-1}{2\Ph},\\
\b_3(t;\la)&=-\frac{\Ph}{2\Ph^2-t\Ph-\la-1}-\frac{\la+1}{2\Ph},\\
\b_4(t;\la)&=\frac{t}{2(\la+2)}+\frac{\Ph}{2\Ph^2-t\Ph-\la-1} 
+\frac{(\la+1)(t^2+2\la+4)\Ph+(\la+1)^2t}{2(\la+2)[2(\la+2)\Ph^2-(\la+1) t\Ph-(\la+1)^2]},
\end{split}\]
where
\begin{align}\Ph(t)&=\deriv{}{t}\ln\left\{\WhitD{-\la-1}\big(-\tfrac12\sqrt2\,t\big)\exp\left(\tfrac18t^2\right)\right\}
=\tfrac12t+\tfrac12\sqrt{2}\,\frac{\WhitD{-\la}\big(-\tfrac12\sqrt2\,t\big)}{\WhitD{-\la-1}\big(-\tfrac12\sqrt2\,t\big)}.
\label{def:Phi}\end{align}
\comment{As $t\to \infty$
\begin{equation*}\Ph(t)=\tfrac12 t+\frac{\la}{t} +\frac{2\la(1-\la)}{t^3}+\frac{4\la(\la-1)(2\la-3)}{t^5}+\mathcal{O}\big(t^{-7}\big).\end{equation*}}
We note that
\[\Ph(t)=\deriv{}{t}\ln\phi_\la(t),\]
where $\phi_\la(t)=\exp(\tfrac18t^2)D_{-\la-1}\left(-\tfrac12\sqrt{2}\,t\right)$ satisfies
\begin{equation*}
\deriv[2]{\phi_\la}{t}-\tfrac12t\deriv{\phi_\la}{t}-\tfrac12(\la+1)\phi_\la=0,
\end{equation*}
and $\Ph(t)$ satisfies the Riccati equation
\begin{equation*}
\deriv{\Ph}{t}=-\Ph^2+\tfrac12t\Ph+\tfrac12(\la+1)=0,
\end{equation*}
see Lemma \ref{lem42}.

Using the recurrence relation \eqref{eq:gfrr}, the first few polynomials are
\[\begin{split}
S_1(x;t,\la)&={x},\\
S_2(x;t,\la)&=x^2-\Ph,\\
S_3(x;t,\la)&=x^3+\frac {t\Ph +\la+1}{2\Ph}\,x,\\
S_4(x;t,\la)&=x^4+{\frac {2t\Ph^{2}-({t}^{2}+2)\Ph -(\la+1) t}{ 2(2\Ph^{2}-t\Ph -\la-1)}}\,x^2 
-{\frac {2(\la+2)\Ph^2-(\la+1) t\Ph-(\la+1)^2}{2(2\Ph^{2}-t\Ph -\la-1)}},\\
S_5(x;t,\la)&=x^5-{\frac {2(\la+3)t\Ph^2-(\la+1) (t^2-2)\Ph-(\la+1)^2t}{4(\la+2)\Ph^2-2(\la+1) t\Ph-2(\la+1)^2}}\,x^3\\
&\qquad -{\frac {\big[2(\la+2)^2-t^2\big]\Ph^{2}-(\la+1)(\la+4)t\Ph -(\la+1)^2(\la+3)}{4(\la+2)\Ph^2-2(\la+1) t\Ph-2(\la+1)^2}}\,x.\\
\end{split}\]

\begin{lemma}{\label{lem42}The function $\Ph(t)$ defined by \eqref{def:Phi} satisfies the Riccati equation
\begin{equation} \label{eq:Phi}\deriv{\Ph}{t}=-\Ph^2+\tfrac12t\Ph+\tfrac12(\la+1),\end{equation}
and has the asymptotic expansion as $t\to\infty$ 
\begin{equation}\label{Phi:expan}\Ph(t)=\tfrac12t+\sum_{n=1}^{\infty}\frac{a_n}{t^{2n-1}},\end{equation}
where the constants $a_n$ are given by the nonlinear recurrence relation
\[a_{n+1}=2(2n-1)a_{n}-2\sum_{j=1}^{n}a_ja_{n+1-j},\]
with $a_1=\la$. In particular, as $t\to\infty$
\begin{equation}\label{Phasyn}
\Ph(t)=\tfrac12 t+\frac{\la}{t} +\frac{2\la(1-\la)}{t^3}+\frac{4\la(\la-1)(2\la-3)}{t^5}+\mathcal{O}\big(t^{-7}\big).\end{equation}
}\end{lemma}

\begin{proof}Letting $\ds\Ph(t)=\deriv{}{t}\ln \phi_{\la}(t)$ in \eqref{eq:Phi} yields
\[\deriv[2]{\phi_{\la}}{t}-\tfrac12t\deriv{\phi_{\la}}{t}-\tfrac12(\la+1)\phi_{\la}=0,\]
which has solution
\[ \phi_{\la}(t)=\left\{C_1\WhitD{-\la-1}\big(-\tfrac12\sqrt2\,t\big)+C_2\WhitD{-\la-1}\big(\tfrac12\sqrt2\,t\big)\right\}\exp(\tfrac18t^2),\]
with $C_1$ and $C_2$ arbitrary constants. Hence setting $C_1=1$ and $C_2=0$ gives the solution \eqref{def:Phi} and shows that $\Ph(t)$ satisfies \eqref{eq:Phi}.

Substituting \eqref{Phi:expan} into \eqref{eq:Phi} gives
\begin{align*}
\sum_{n=1}^{\infty}\frac{(2n-1)a_n}{t^{2n}} &=\tfrac12t \sum_{n=1}^{\infty}\frac{a_n}{t^{2n-1}}
+\left(\sum_{n=1}^{\infty}\frac{a_n}{t^{2n-1}}\right)^2-\tfrac12\la\\
&=a_1-\la+\tfrac12\sum_{n=1}^{\infty}\frac{a_{n+1}}{t^{2n}}+\sum_{n=1}^{\infty}\frac{1}{t^{2n}}\sum_{j=1}^{n}a_ja_{n+1-j},
\end{align*}
hence, comparing coefficients of powers of $t$ gives $a_1=\la$ and
\[a_{n+1}=2(2n-1)a_{n}-2\sum_{j=1}^{n}a_ja_{n+1-j},\]
as required. Hence
\[\begin{split}a_1&=\la,\qquad a_2=-2\la( \la-1),\qquad a_3=4\la(\la-1)(2\la-3), 
\end{split}\]
which gives \eqref{Phasyn} as required.
 \end{proof}

\begin{lemma}{Let $h_n(t;\la)$ be defined by
\begin{equation}\label{def:Hn}
H_n(t;\la) = \deriv{}{t}\ln\tau_n(t;\la),
\end{equation}
where $\tau_n(t;\la)$ is the Wronskian given by
\[
\tau_n(t;\la)=\mathcal{W}\left(\phi_{\la},\deriv{\phi_{\la}}{t},\ldots,\deriv[n-1]{\phi_{\la}}{t}\right),\]
with \begin{equation*}
\phi_{\la}(t)=\frac{\Gamma(\la+1)}{2^{(\la+1)/2}}\,\exp\left(\tfrac18t^2\right)\WhitD{-\la-1}\big(-\tfrac12\sqrt2\,t\big).
\end{equation*}
Then $H_n(t;\la)$ satisfies the second-order, second-degree equation
\begin{align}
\left(\deriv[2]{H_n}{t}\right)^2&-\tfrac14\left(t\deriv{H_n}t-H_n\right)^2
+\deriv{H_n}t\left(2\deriv{H_n}t-n\right)\left(2\deriv{H_n}t-n-\la\right)=0.\label{eq:Hn}
\end{align}
}\end{lemma}
\begin{proof}{Equation \eqref{eq:Hn} is equivalent to \sPIV, 
the \PIV\ $\sigma$-equation
\begin{equation}\label{eq:sPIV}
\left(\deriv[2]{\sigma}z\right)^{2} - 4\left(z\deriv{\sigma}z-\sigma\right)^{2} +4\deriv{\sigma}z\left(\deriv{\sigma}z+2\theta_0\right)
\left(\deriv{\sigma}z+2\theta_{\infty}\right)=0,\end{equation}
as shown in \cite[Theorem 4.11]{refCJ}. Equation \eqref{eq:Hn} is the same as equation (4.15) in \cite{refCJ}. Special function solutions of \sPIV\ \eqref{eq:sPIV} in terms of parabolic cylinder functions have been classified in \cite{refFW01,refOkamotoPIIPIV}; see also \cite[Theorem 3.5]{refCJ}.
}\end{proof}

We remark that equation \eqref{eq:sPIV}, and hence also equation \eqref{eq:Hn}, is equivalent to equation SD-I.c in the classification of second order, second-degree differential equations with the \p\ property by Cosgrove and Scoufis \cite{refCS}, an equation first derived and solved by Chazy \cite{refChazy11}.

\begin{lemma}{As $t\to\infty$, the recurrence coefficient $\b_n(t)$ has the asymptotic expansion
\begin{subequations}
\begin{align}
\b_{2n}(t;\la)&=\frac{n}{t}-\frac{2n(2\la-n+1)}{t^3}+\mathcal{O}\big(t^{-5}\big),\\ 
\b_{2n+1}(t;\la)&=\frac{t}{2}+\frac{\la-n}{t}-\frac{2(\la^2-4\la n+n^2-\la-n)}{t^3}+\mathcal{O}\big(t^{-5}\big),
\end{align}\end{subequations}
for $n\in\N$.
}\end{lemma}
\begin{proof}{In terms of the function $H_n(t;\la)$ defined by \eqref{def:Hn}, the recurrence coefficients are given as
\begin{subequations}\label{bn:asym}
\begin{align}\b_{2n}(t;\la)  &= H_n(t;\la+1)-H_n(t;\la),\\ \b_{2n+1}(t;\la)  &= H_{n+1}(t;\la)-H_n(t;\la+1).
\end{align}\end{subequations}
As $t\to\infty$, $H_n(t;\la)$ has the asymptotic expansion
\begin{equation}\label{asym:Hn}
H_n(t;\la)=\frac{nt}{2}+{\frac {n\la}{t}}+{\frac {2n\la(n- \la)}{{t}^{3}}}+\mathcal{O}\big(t^{-5}\big),
\end{equation} for $n\in\N$, see \cite[Lemma 5.2]{refCJ}; note that the functions $\Delta_n(t)$ and $S_n(t)$ in \cite{refCJ} are the same as our functions $\tau_n(t;\la)$ and $H_n(t;\la)$, respectively. Substituting \eqref{asym:Hn} in \eqref{bn:asym} immediately gives the result.
}\end{proof}

\subsection{The differential-difference equation satisfied by generalized Freud polynomials}\label{sec:diff-differential}
In this section we derive a differential-difference equation satisfied by generalized Freud polynomials using our results in \S\ref{sec:scOPs}.
\begin{lemma}{\label{lem31}For the generalized Freud weight \eqref{genFreud}
the monic orthogonal polynomials $S_{n}(x;t)$ satisfy
\begin{align}
 \int_{-\infty}^\infty \vvxy&{S_n^2(y;t)}\,\ww{y}\,dy 
 =4\big(x^2-\tfrac12t +\b_n+\b_{n+1}\big)h_n, \label{int1a}\\
\int_{-\infty}^\infty  \vvxy& {S_n(y;t)S_{n-1}(y;t)}\,\ww{y}\,dy = 4xh_n,\label{int1b}
\end{align} where 
\[\vvxy=\frac{\vx-\vy}{x-y},\] with
$v(x;t)=x^4-tx^2$ and
\begin{equation}\label{def:hn} h_n = \int_{-\infty}^\infty  {S_n^2(y;t)\ww{y}}\,dy.\end{equation}}\end{lemma}

\begin{proof}For the generalized Freud weight \eqref{genFreud} we have
\[\ww{x}=|x|^{2\la+1}\exp\left(-x^4+tx^2\right),\] i.e.\
$v(x;t)=x^4-tx^2$, and so
\[\vvxy= 4x^2+4xy+4y^2-2t.\]
Hence for \eqref{int1a}
\begin{align*} \int_{-\infty}^\infty  \vvxy {S_n^2(y;t)}\,\ww{y}\,dy &=
(4x^2-2t)\int_{-\infty}^\infty  {S_n^2(y;t)}\,\ww{y}\,dy \\&\qquad  + 4x\int_{-\infty}^\infty  {yS_n^2(y;t)}\,\ww{y}\,dy + 4\int_{-\infty}^\infty  {y^2S_n^2(y;t)}\,\ww{y}\,dy\\
&= (4x^2-2t)h_n +4x\int_{-\infty}^\infty  {S_n(y;t)}\big[S_{n+1}(y;t)+\b_nS_{n-1}(y;t)\big]\ww{y}\,dy\\ &\qquad
+ 4\int_{-\infty}^\infty  \big[S_{n+1}(y;t) + \b_nS_{n-1}(y;t)\big]^2\ww{y}\,dy\\
&= (4x^2-2t)h_n + 4h_{n+1}+ 4\b_n^2h_{n-1}\\
&= 4\big(x^2-\tfrac12t +\b_n+\b_{n+1}\big)h_n,
\end{align*}
since $\b_n=h_n/h_{n-1}$, the monic orthogonal polynomials $S_{n}(x;t)$ satisfy the three-term recurrence relation \eqref{eq:gfrr},
and are orthogonal, i.e.
\begin{equation}\int_{-\infty}^\infty  {S_m(y;t)}S_{n}(y;t)\ww{y}\,dy=0,\qquad{\rm if}\quad m\neq n.\label{Pnorth}\end{equation}
Also for \eqref{int1b}
\begin{align*} \int_{-\infty}^\infty  \vvxy &{S_n(y;t)S_{n-1}(y;t)}\,\ww{y}\,dy\nonumber\\ &=
(4x^2-2t) \int_{-\infty}^\infty  {S_n(y;t)S_{n-1}(y;t)}\,\ww{y}\,dy 
+ 4x \int_{-\infty}^\infty  {yS_n(y;t)S_{n-1}(y;t)}\,\ww{y}\,dy \\ &\qquad
+ 4 \int_{-\infty}^\infty  {y^2S_n(y;t)S_{n-1}(y;t)}\,\ww{y}\,dy \\
&= 4x \int_{-\infty}^\infty  S_{n}(y;t)\big[S_{n}(y;t)+\b_{n-1}S_{n-2}(y;t)\big] \ww{y}\,dy \\
&\qquad +4 \int_{-\infty}^\infty  \big[S_{n+1}(y;t)+\b_nS_{n-1}(y;t)\big] 
\big[S_{n}(y;t)+\b_{n-2}S_{n-2}(y;t)\big]\ww{y}\,dy\\
 &=4xh_n,
\end{align*}
using the recurrence relation \eqref{eq:gfrr} and orthogonality \eqref{Pnorth}.\end{proof}

\begin{theorem}{For the generalized Freud weight \eqref{genFreud}
the monic orthogonal polynomials $S_{n}(x;t)$ satisfy the differential-difference equation
\begin{equation}\label{eq:Snddeq}
x\deriv{S_n}{x}(x;t)=-B_n(x;t)S_n(x;t)+A_n(x;t)S_{n-1}(x;t),
\end{equation}
where
\begin{subequations}\label{AnBn}\begin{align}
A_n(x;t)&=4x\b_n(x^2-\tfrac12t+\b_n+\b_{n+1}),\\
B_n(x;t)&=4x^2\b_n+(\la+\tfrac12)[1-(-1)^n],
\end{align}\end{subequations}
with $\b_n$ the recurrence coefficient in the three-term recurrence relation \eqref{eq:gfrr}.}\end{theorem}

\begin{proof}Corollary \ref{cor:ABn} shows that monic orthogonal polynomials $S_{n}(x;t)$ with respect to the weight $$\ww{x}=|x|^{2\la+1}\exp\{-v(x;t)\},$$ satisfy the differential-difference equation \eqref{eq:Snddeq}, where
\begin{subequations}\label{def:AnBn}
\begin{align}
A_n(x;t)&=\frac{x}{h_{n-1}}\int_{-\infty}^\infty \vvxy {S_n^2(y;t)}w(y;t)\,dy,\\
B_n(x;t)&=\frac{x}{h_{n-1}}\int_{-\infty}^\infty \vvxy {S_n(y;t)S_{n-1}(y;t)} w(y;t)\,dy 
+(\la+\tfrac12)[1+(-1)^n] .
\end{align}\end{subequations}
For the generalized Freud weight \eqref{genFreud}, 
using Lemma \ref{lem31} yields the result.\end{proof}

\subsection{The differential equation satisfied by generalized Freud polynomials}\label{sec:diffeqn}
Now we derive a differential equation satisfied by generalized Freud polynomials.
\begin{theorem}{\label{thm43}For the generalized Freud weight \eqref{genFreud}
the monic orthogonal polynomials $S_{n}(x;t)$ satisfy the differential equation
\begin{equation}\label{eq:Snode}
x\deriv[2]{S_n}{x}(x;t)+R_n(x;t)\deriv{S_n}{x}(x;t)+T_n(x;t)S_n(x;t)=0,
\end{equation}
where
\begin{subequations}\begin{align}
R_n(x;t) 
&=-4x^4+2tx^2+{2\la+1}-\frac{2x^2}{x^2-\tfrac12t+\b_n+\b_{n+1}},\\[5pt]
T_n(x;t) 
&= 4nx^3+16x\b_n(\b_n+\b_{n+1}-\tfrac12t)(\b_n+\b_{n-1}-\tfrac12t)\nonumber\\
&\qquad+4x[1+(2\la+1)(-1)^n]\b_n -\frac{8\b_n x^3+(2\la+1)[1-(-1)^n]x}{x^2-\tfrac12t+\b_n+\b_{n+1}}\nonumber\\
&\qquad+(2\la+1)[1-(-1)^n]x\left(t-\frac{1}{2x^2}\right).
\end{align}\end{subequations}
}\end{theorem}

\begin{proof}
In Theorem \ref{thm:gende} we proved that the coefficients in the differential equation \eqref{eq:gende} satisfied by polynomials orthogonal with respect to the weight \[\ww{x}=|x-k|^{\ga}\exp\{-\v(x;t)\},\] are given by  \eqref{coef1} and \eqref{coef2}. For the generalized Freud weight \eqref{genFreud} we use \eqref{coef1} and \eqref{coef2} with $k=0$, $\ga=2\la+1$, $\v(x;t)=x^4-tx^2$, and $A_n$ and $B_n$ given by \eqref{AnBn} to obtain the stated result.
\end{proof}

\begin{remark}We note that if $\{P_n(x)\}_{n=0}^{\infty}$, is a sequence of \textit{classical} orthogonal polynomials (such as Hermite, Laguerre and Jacobi polynomials), then $P_n(x)$ satisfies second-order ordinary differential equation
\begin{equation}\label{eq:Pn}
\sigma(x)\deriv[2]{P_n}{x}+\tau(x)\deriv{P_n}{x}=\la_nP_n,
\end{equation}
where $\sigma(x)$ is a monic polynomial with deg$(\sigma)\leq2$, $\tau(x)$ is a polynomial with deg$(\tau)=1$, and $\la_n$ is a real number which depends on the degree of the polynomial solution, see Bochner \cite{refBochner}. For {classical} orthogonal polynomials, the polynomials $\sigma(x)$ and $\tau(x)$ are the same as in the associated Pearson equation \eqref{eq:Pearson}. In contrast the coefficients in second-order ordinary differential equation satisfied by the polynomials for the generalized Freud weight given in Theorem \ref{thm43} are not the same as the polynomials $\sigma(x)=x$ and $\tau(x)=-4x^4+2tx^2+2\la+1$ in the Pearson equation \eqref{eq:Pearson} 
satisfied by the generalized Freud polynomials since
\begin{equation}\nonumber
\dfrac{w'(x;t)}{w(x;t)} =\frac{\tau(x)-\sigma'(x)}{\sigma(x)}=-4x^{3}+2tx + \dfrac{2\la+1}{x}, \quad x\in \R\backslash \lbrace 0\rbrace.
\end{equation}\end{remark}

\subsection{An alternative method due to Shohat}\label{sho}
It is shown in \cite{Nevai} and  \cite{refShohat39} that the monic orthogonal polynomials $S_{n}(x;t)$ orthogonal with respect to the generalized Freud weight \eqref{genFreud} are quasi-orthogonal of order $m=5$ and hence we can write
\begin{align}\label{quasi}
x\deriv{S_n}{x}(x;t) &=  \sum_{k=n-4}^{n}c_{n,k} S_k(x;t),
\end{align}
where the coefficient $c_{n,k}$ is given by
\begin{align}\label{Ccoeff}
c_{n,k} &= \frac{1}{h_{k}}  \int_{-\infty}^{\infty}   x\deriv{S_n}{x}(x;t) S_k(x;t) \ww{x}  \,d{x},
\end{align}
for $n-4 \leq k\leq n$ and $h_k\neq 0$.

Integrating by parts, we obtain for $n-4 \leq k\leq n-1$,
\begin{align}\nonumber
 h_k c_{n,k} &=\Big[ xS_k(x;t) S_n(x;t) w(x;t)\Big]_{-\infty}^{\infty}  
 - \int_{-\infty}^\infty \deriv{}{x}\left[ x S_k (x;t) \ww{x} \right]S_n(x;t)  \,d{x}  \nonumber\\
&=- \int_{-\infty}^{\infty}   \left[S_n(x;t) S_k(x;t)  
  + xS_n(x;t) \deriv{S_k}{x}(x;t)\right] w(x;t) \,d{x}
   - \int_{-\infty}^{\infty}   xS_n(x;t) S_k(x;t) \deriv{w}{x}(x;t) \,d{x},
\nonumber\\
 &= \nonumber    - \int_{-\infty}^{\infty} {  xS_n(x;t) S_k(x;t)  \deriv{w}{x}(x;t)}\,d{x}\\
&=    -\int_{-\infty}^{\infty} {S_n(x;t) S_k(x;t) \left( -4x^{4} +2tx^2 +2\la +1\right) }   \ww{x}\,d{x}\nonumber
 \\&= 
  \int_{-\infty}^{\infty} \big(4x^4-2tx^2\big) S_n(x;t) S_k(x;t) \ww{x} \,d{x},\label{maineq}
 \end{align}
 since
 \[ x\deriv{w}{x}(x;t) =\big(-4x^4+2x^2 + 2\la+1\big)\ww{x}.\]
Iterating the three-term recurrence relation \eqref{eq:gfrr}, the following relations are obtained
\begin{subequations}\label{recurrence13}
\comment{\begin{align}
x^2 S_n(x;t) &= S_{n+2}(x;t) +( \b_{n} + \b_{n+1}) S_{n}(x;t) + \b_{n} \b_{n-1} S_{n-2}(x;t),\label{recurrence1}\\
x^4 S_n(x;t)& = S_{n+4}(x;t) + (  \b_{n}+ \b_{n+1} + \b_{n+2} +\b_{n+3})  S_{n+2}(x;t)
\nonumber\\& + [\b_{n} ( \b_{n-1} + \b_{n}  + \b_{n+1})    +\b_{n+1} ( \b_{n} + \b_{n+1}  + \b_{n+2} )]  S_{n}(x;t)
\nonumber\\& +\b_{n}  \b_{n-1}(  \b_{n-2} +\b_{n-1} + \b_{n}  + \b_{n+1} ) S_{n-2}(x;t)
 \nonumber \\&
+ ( \b_{n}  \b_{n-1} \b_{n-2} \b_{n-3} )  S_{n-4}(x;t).\label{recurrence3}
\end{align}}
\begin{align}
x^2 S_n = S_{n+2} &+( \b_{n} + \b_{n+1}) S_{n} + \b_{n} \b_{n-1} S_{n-2},\label{recurrence1}\\
x^4 S_n = S_{n+4} &+ (  \b_{n}+ \b_{n+1} + \b_{n+2} +\b_{n+3})  S_{n+2}
\nonumber\\& 
+ \big[\b_{n} ( \b_{n-1} + \b_{n}  + \b_{n+1})    +\b_{n+1} ( \b_{n} + \b_{n+1}  + \b_{n+2} )\big]  S_{n}
\nonumber\\& +\b_{n}  \b_{n-1}(  \b_{n-2} +\b_{n-1} + \b_{n}  + \b_{n+1} ) S_{n-2}
+ ( \b_{n}  \b_{n-1} \b_{n-2} \b_{n-3} )  S_{n-4}.\label{recurrence3}
\end{align}

\end{subequations}
Substituting \eqref{recurrence13} into \eqref{maineq} yields the coefficients 
$\lbrace c_{n,k} \rbrace_{k=n-4}^{n-1}$ in \eqref{quasi}.
\begin{subequations}\label{Aau}
\begin{align} 
c_{n,n-4} &  = 4 \b_{n}  \b_{n-1} \b_{n-2}  \b_{n-3},\\
c_{n,n-3} & =0,\\
c_{n,n-2} &= 4\b_{n} \b_{n-1} (\b_{n-2} +  \b_{n-1} + \b_{n}  + \b_{n+1}- \tfrac12t),\\
 c_{n,n-1} & =0.
 \end{align}
\end{subequations}
Lastly we consider the case when $k=n$. Integration by parts   in \eqref{Ccoeff} yields
\begin{align}
h_n c_{n,n}  &=\int_{-\infty}^{\infty}   x\deriv{S_n}{x}(x;t) S_n(x;t) \ww{x}  \,d{x}\nonumber\\
&=-\tfrac12\int_{-\infty}^{\infty}   S_n^2(x;t) \left[\ww{x} + x\deriv{w}{x}(x;t)\right] d{x}
\nonumber\\&= - \tfrac12h_n 
   +\int_{-\infty}^{\infty}   S_n^2(x;t)\big(2x^{4} -tx^2 -\la -\tfrac12\big)\,\ww{x} \,d{x}
 \nonumber\\&=  \int_{-\infty}^{\infty} \big(2x^4-tx^2\big) S_n^2(x;t)\,\ww{x} \,d{x} - (\la+1) h_{n}. \label{eq:cnn1}
 \end{align}
 From the three-term recurrence relation \eqref{eq:gfrr}, we have
 \begin{align*} x^2 S_n^2&=(S_{n+1}+\b_n S_{n-1})^2 \\
 &= S_{n+1}^2+2\b_nS_{n+1}S_{n-1}+\b_n^2S_{n-1}^2\\
 x^4 S_n^2&= x^2\big(S_{n+1}^2+2\b_nS_{n+1}S_{n-1}+\b_n^2S_{n-1}^2)\\
 &= x^2 S_{n+1}^2 +2\b_n(xS_{n+1})(xS_{n-1})+\b_n^2x^2S_{n-1}^2\\
 &= \big(S_{n+2}+\b_{n+1}S_{n}\big)^2 + 2\b_n\big(S_{n+2}+\b_{n+1}S_{n}\big)\big(S_{n}+\b_{n-1}S_{n-2}\big)
+\b_n^2\big(S_{n}+\b_{n-1}S_{n-2}\big) ^2\\ 
&= S_{n+2}^2+2(\b_{n+1}+\b_n)S_{n+2}S_{n}+(\b_{n+1}+\b_n)^2 S_{n}^2+2\b_n\b_{n-1}S_{n+2}S_{n-2}
\\ &\qquad\quad
+2\b_n\b_{n-1}(\b_n+\b_{n+1})S_{n}S_{n-2}+\b_n^2\b_{n-1}^2S_{n-2}^2
  \end{align*}
and so by orthogonality 
 \begin{align}
 \int_{-\infty}^{\infty}  x^2 S_n^2(x;t)\,\ww{x} \,d{x} \
&= h_{n+1}+\b_n^2 h_{n-1} = (\b_{n+1}+\b_{n})h_{n},\label{eq:cnn2} \\
 \int_{-\infty}^{\infty}  x^4 S_n^2(x;t)\ww{x} \,d{x} &= h_{n+2}+(\b_{n+1}+\b_n)^2h_n +\b_n^2\b_{n-1}^2h_{n-2}\nonumber\\
 &=  \b_{n+2}\b_{n+1}h_n+(\b_{n+1}+\b_n)^2h_n +\b_n\b_{n-1}h_n \nonumber\\
 &=\big[(\b_{n+1}+\b_{n}+\b_{n-1})\b_n 
 + (\b_{n+2}+\b_{n+1}+\b_n)\b_{n+1}\big]h_n\nonumber\\
 &=\tfrac12 \big[t(\b_{n+1}+\b_{n})+n+\la+1\big]h_{n},\label{eq:cnn3}
  \end{align}
using $h_{n+1}=\b_{n+1}h_n$ and \dPI\ \eqref{eq:dPI}. Hence from \eqref{eq:cnn1}, \eqref{eq:cnn2} and \eqref{eq:cnn3} we have
\begin{align} c_{n,n}&= t(\b_{n+1}+\b_{n})+n+\la+1-t(\b_{n+1}+\b_{n})- (\la+1)\nonumber\\
&=n. \label{eq:cnn4}\end{align}

Combining \eqref{Aau} with \eqref{quasi}, we write
 \begin{align}\label{AQquasi}
x\deriv{S_n}{x}(x;t) = c_{n,n-4} S_{n-4}(x;t) +  c_{n,n-2} S_{n-2}(x;t)  +  c_{n,n} S_{n}(x;t).
\end{align}
In order to express  $S_{n-4}$ and $S_{n-2}$ in \eqref{AQquasi} in terms of $S_n$ and $S_{n-1}$, we iterate \eqref{eq:gfrr}  to obtain
\begin{align}\label{recoa}
S_{n-2}&= \dfrac{xS_{n-1}-S_n}{\b_{n-1}}, \\
S_{n-3}&= \dfrac{xS_{n-2}-S_{n-1}}{\b_{n-2}} 
= \dfrac{x^2-\b_{n-1}}{\b_{n-1}\b_{n-2}}\, S_{n-1} - \dfrac{x}{\b_{n-1}\b_{n-2}}\, S_{n},\\ 
S_{n-4}
&= \dfrac{xS_{n-3} - S_{n-2}}{\b_{n-3}} 
= \label{recoab}   \dfrac{x^3-  ( \b_{n-1}+\b_{n-2})x}{\b_{n-1}\b_{n-2}\b_{n-3}}\,  S_{n-1} -
    \dfrac{x^2-\b_{n-2}}{ \b_{n-1}\b_{n-2}\b_{n-3}} \,  S_{n}.
   \end{align}
Substituting \eqref{Aau},  \eqref{eq:cnn4}, \eqref{recoa} and \eqref{recoab} into \eqref{AQquasi} yields
\begin{align}\label{Aomega}
x\deriv{S_n}{x}(x;t)= -B_n(x)  S_n(x;t) + A_n(x;t)  S_{n-1}(x;t),
\end{align}
where $A_n(x)$ and $B_n(x)$ are given by \eqref{AnBn}.
 
\section{Conclusion}\label{sec:concl}
In this paper, for the generalized Freud weight \eqref{genFreud} we have obtained explicit expressions for the coefficients of the three-term recurrence relation and differential-difference equation satisfied by generalized Freud polynomials. We also proved that the generalized Freud polynomials satisfy a linear ordinary differential equation. We note that the closed form expressions for the coefficients provided allow investigation of other properties, including properties of the zeros such as monotonicity, convexity and inequalities satisfied by the zeros. However, although the expressions for the coefficients given in this paper are explicit, they are rather complicated and given in terms of special function solutions of the fourth \p\ equation which does not necessarily lead to elegant results in applications. For this reason, a natural extension of this work would be an investigation of asymptotic properties using limiting relations satisfied by the polynomials as the parameters $t$ and/or $\la$ tend to $\infty$.

\bigskip
\section*{Acknowledgments}
We thank the London Mathematical Society for support through a ``Research in Pairs" grant. PAC thanks Alexander Its, Ana Loureiro, and Walter van Assche for their helpful comments and illuminating discussions.
KJ thanks the School of Mathematics, Statistics \& Actuarial Science at the University of Kent for their hospitality during her visit when some of this research was done. The research by KJ was partially supported by the National Research Foundation of South Africa.

\def\p{Painlev\'{e}}
\def\JPA{J. Phys. A}

\def\refjl#1#2#3#4#5#6#7{\vspace{-0.25cm}
\bibitem{#1}{\frenchspacing#2}, \textrm{#3},
\textit{\frenchspacing#4}, \textbf{#5} (#7) #6.}

\def\refjltoap#1#2#3#4#5#6#7{\vspace{-0.25cm}
\bibitem{#1} {\frenchspacing#2}, \textrm{#3},
\textit{\frenchspacing#4}, 
#6.}

\def\refbk#1#2#3#4#5{\vspace{-0.25cm}
\bibitem{#1} {\frenchspacing#2}, \textit{#3}, #4, #5.}

\def\refcf#1#2#3#4#5#6{\vspace{-0.25cm}
\bibitem{#1} {\frenchspacing#2}, \textrm{#3},
in \textit{#4}, {\frenchspacing#5}, #6.}

\def\and{\mbox{\rm and}\ }

\end{document}